\documentclass{amsart}
\usepackage[utf8]{inputenc}
\usepackage[T1]{fontenc}
\usepackage{amsmath,amssymb,amsfonts,mathtools,amsthm}
\usepackage{color}
\usepackage{esint}
\usepackage{enumerate}
\usepackage[colorlinks]{hyperref}

\usepackage{verbatim}

\theoremstyle{plain}
\newtheorem{theorem}{Theorem}
\newtheorem{lemma}{Lemma}
\newtheorem{proposition}{Proposition}
\newtheorem{corollary}[theorem]{Corollary}
\theoremstyle{definition}
\newtheorem{claim}{Claim}
\newtheorem{remark}{Remark}

\def\R{\mathbb{R}}
\def\N{\mathbb{N}}

\def\X{\mathcal{X}}
\def\Td{\mathbb{T}^d}
\def\div{\operatorname{div}}

\def\tox{\xrightarrow{\X}}

\def\re{\mathring{R}}

\title{The Baire category method for intermittent convex integration}
\author{Gabriel Sattig}
\address{Mathematisches Institut, Universit\"at Leipzig, D-04109 Leipzig, Germany}
\email{gabriel.sattig@math.uni-leipzig.de}
\author{L\'aszl\'o Sz\'ekelyhidi}
\address{Mathematisches Institut, Universit\"at Leipzig, D-04109 Leipzig, Germany}
\email{laszlo.szekelyhidi@math.uni-leipzig.de}

\begin{document}
	
	\begin{abstract}
		We use a convex integration construction from \cite{ModenaSattig2020} in a Baire category argument to show that weak solutions to the transport equation with incompressible vector fields with Sobolev regularity are generic in the Baire category sense. Using the construction of \cite{BurczakModenaSzekelyhidi20} we prove an analog statement for the 3D Navier-Stokes equations.
	\end{abstract}
	
	\maketitle
	
	\section{Introduction}
	\label{sec:intro}
	
	In the past 15 years convex integration has been established as a very versatile and flexible method for the construction of weak solutions to various equations in fluid mechanics. In this context the first results appeared for the construction of bounded weak solutions of the incompressible Euler equations \cite{DeLellisSzekelyhidi09,DeLellisSzekelyhidi10}, with a method of construction directly related to the construction of Lipschitz solutions to first order partial differential inclusions \cite{dacorogna_marcellini99,Kirchheim:2002wc,muller_sverak03}. In particular, following the ideas developed by Cellina \cite{cellina2005}, Dacorogna-Marcellini \cite{dacorogna_marcellini97} and Kirchheim \cite{kirchheim03}, in this setting one can reduce the existence proof to a simple perturbation statement. Subsequently, similar ideas were applied to the compressible Euler system as well as certain models involving an active scalar. For a comprehensive survey of related results we refer to \cite{MR2917063}. 
	
	At the same time, motivated by Onsager's conjecture, a different framework was developed in \cite{DeLellisSzekelyhidi13} for the construction of continuous and H\"older continuous weak solutions of the incompressible Euler system. In this method of construction, which is closely related to the construction of $C^1$ isometric embeddings by Nash and Kuiper \cite{nash54} (see the survey \cite{delellis_szekelyhidi16} for connections between the Nash-Kuiper theorem and fluid mechanics), it seems necessary to explicitly construct an approximating sequence with the required estimates and convergence properties, a simple perturbation step seems insufficient. In particular, the final weak solution does not lie in a natural, easily identifiable function space, but rather the optimal regularity properties heavily depend on the particular equation and the delicate estimates involved in the construction. In the case of the Euler equations the optimal H\"older exponent is dictated by Onsager's conjecture and the conservation of energy, and this was reached in subsequent work of Isett \cite{Ise19,BDLSV19}). In other examples, for instance incompressible Euler equations with the local energy inequality, compressible Euler system, the SQG system, or the Nash-Kuiper isometric embedding problem, current constructions are not strong enough to reach conjectured optimal exponents. 
	
	In a further groundbreaking development, Buckmaster and Vicol modified the constructive scheme used for Euler in order to construct weak solutions to the Navier-Stokes equations \cite{BuckmasterVicol19}. Their technique resembles the construction of Euler flows with high Hölder regularity, but uses as building blocks functions which are very much concentrated in space, which allows to control the dissipative term. This high concentration is referred to as "intermittency", in distant analogy with the phenomenon and mathematical theory of intermittency in turbulence. Subsequently, similar constructions were used to construct solutions for the transport equation with Sobolev vector fields \cite{ModenaSzekelyhidi18,ModenaSzekelyhidi19,ModenaSattig2020,BrueColomboDeLellis21,PitchoSorella21,BuckModena22} and Navier-Stokes equations for non-Newtonian fluids \cite{BurczakModenaSzekelyhidi20}.
	
	In this short note we wish to point out that these latter, "intermittent" constructions are actually closer in spirit to the $L^\infty$-framework in the sense that a simple perturbation step combined with a standard Baire category argument suffices for showing existence of weak solutions - an explicit construction of an approximating sequence is not necessary and the there is a natural space in which the limit objects of such constructions lie. 
	
	More precisely, we will consider the construction of weak solutions of the linear transport equation as well as weak solutions of the incompressible Navier-Stokes equations. We wish to emphasize that the existence statements proved here are not new. It is merely that our method of proof is simpler, relying on a perturbation argument rather than an explicit construction of an approximating sequence. 
	
	\subsection{The transport equation}
    We prove genericity of weak solutions to the transport equation
	\begin{equation}	\label{eq:transport}
	\begin{split}
	\partial_t \rho + v\cdot\nabla\rho &=0 \\
	\div v &=0
	\end{split}
	\end{equation}
	with divergence free velocity field $v$ in some Sobolev space and density in a Lebesgue space.
	Non-uniqueness of such solutions in the very same setting (which is motivated by its similarity with the setting in the seminal work of DiPerna and Lions \cite{DiPernaLions89}) has been shown in a line of research \cite{ModenaSzekelyhidi18,ModenaSzekelyhidi19,ModenaSattig2020} by the authors and S.~Modena by the constructive method of intermittent convex integration. The aim of this note is to give a new proof which is not constructive but instead uses a Baire category argument.
	While this implies genericity of such solutions which was not known, our main motivation is to show that the implicit Baire category method can be combined with intermittent convex integration, even if high Sobolev regularity of solutions is desired.
	
	We will prove the following
	\begin{theorem}
		\label{thm:transport}
		If $p\in(1,\infty)$ and $\tilde{p}\ge 1$ such that
		\begin{equation}\label{eq:r-condition}
		\frac{1}{p} + \frac{1}{\tilde{p}} > 1+\frac{1}{d}
		\end{equation}
		then the set of solutions to \eqref{eq:transport} in the class
		\begin{equation*} 
		\rho\in C\left( [0,T],L^p_x(\Td) \right) , \ u\in C\left( [0,T],L^{p'}_x(\Td) \cap W^{1,\tilde{p}}_x(\Td) \right)
		\end{equation*}
		is generic (in the Baire sense).
\end{theorem}

The statement of the theorem is deliberately vague here, because the precise statement (Theorem~\ref{thm:main}) requires some notation and definitions.

\subsection{Navier-Stokes equations}

We will then proceed to show that the same strategy can also be applied to the d-dimensional Navier-Stokes equations for $d\geq 3$. More precisely, we will show
	\begin{theorem}
		\label{thm:ns}
		If $r\in(1,\infty)$ is such that
		\begin{equation}\label{eq:p-q-condition}
		\frac{1}{2} + \frac{1}{r} > 1+\frac{1}{d}
		\end{equation}
		then the set of solutions to \eqref{eq:ns} in the class
		\begin{equation*} 
		v\in C\left( [0,T],L^{2}_x(\Td) \cap W^{1,r}_x(\Td) \right)
		\end{equation*}
		is generic (in the Baire sense).
	\end{theorem}
	
The more precise statement and its proof will be given in Section \ref{sec:ns}.

\subsection{Outline of the Baire category method}
\label{subsec:baire-outline}

By the term Baire method we mean the following framework for proving genericity (and along the way existence) of a class of objects with the desired properties. Its ingredients are three objects:
\begin{enumerate}
    \item a set $Y_0$ of `good objects',
    \item a metric $d$ on $Y_0$ and the completion $Y=\overline{Y_0}$ with respect to this metric
    \item a functional $F$ on $Y$ which is of Baire class 1, i.e.~pointwise limit of functionals which are continuous with respect to~$d$.
\end{enumerate}
Furthermore it requires two simple statements:
\begin{enumerate}[(i)]
    \item if $F(y)=0$ for some $y\in Y$ then $y$ has the desired properties,
    \item if $F(y)\ne 0$ then $F$ is discontinuous at $y$.
\end{enumerate}
Since the points of continuity of a map of Baire class 1 form a residual set it follows that the set $\{y\in Y \text{ s.th.~$y$ has the desired properties}\}$ is a residual set in $Y$, i.e.~its complement is a meager set, a countable union of nowhere dense sets. In particular, since $Y$ is a complete metric space, Baire's theorem implies that the desired objects are dense in $Y$, and there are infinitely many of them, if the cardinality of $Y_0\subset Y$ is infinite.

Here and in other convex integration results $Y_0$ is the space of smooth `subsolutions', i.e.~solutions of a relaxed equation, and $d$ is a metric inducing the weak or weak-* topology on a bounded subset of a Banach space. There is some freedom in the choice of the functional; here it will be analog to the kinetic energy in the Euler case. The statement (i) follows directly form the `functional setup', i.e.~the choice of the objects mentioned above, while (ii) requires the perturbation proposition, an explicit construction.

	The article is organized as follows. In Sections \ref{sec:functional} and \ref{sec:functional-ns} we define the functional setup, state the results in full precision and apply the Baire category method to reduce it to a quantitative perturbation property (Propositions \ref{prop:lifting} and \ref{prop:perturbation-ns}). In turn, these will be proved in Sections \ref{sec:perturbation} and \ref{sec:perturbation-ns}, where we heavily refer to the convex integration constructions in  \cite{ModenaSattig2020} (for the transport equation) and \cite{BurczakModenaSzekelyhidi20} (for the Navier-Stokes equations).

	\subsection*{Acknowledgements }
	G.S.~thanks Luigi de Rosa for useful discussions on the topic of this paper.\\
	This work was supported by the European Research Council (ERC) under the European Union’s Horizon 2020 research and innovation programme (grant agreement No.724298-DIFFINCL).

	\section{The Transport Equation}
	\label{sec:transport}
	Throughout this article we assume that $p,\tilde{p}$ fulfil the assumptions of Theorem~\ref{thm:transport}. Any unspecified Lebesgue norm $\|\cdot\|_{L^p}$ refers to the spatial domain $\Td$. 
	\subsection{Functional setup}
	\label{sec:functional}
	\subsubsection{The space of subsolutions and its metric}
	
	Let $e:[0,T]\to \R_{>0}$ be a strictly positive smooth function and $M\geq 1$ a constant to be determined later (see Remark~\ref{rem:C-def}). A subsolution is smooth tuple of functions $(\rho,u):[0,T]\times \Td\to \R\times\R^d$ such that there is a smooth vector field $R:[0,T]\times \Td\to \R^d$ such that the triple $(\rho,u,R)$ solves the transport-defect equation
\begin{subequations}
\begin{equation}\label{eq:trans-def}
\begin{split}
\partial_t \rho + u\cdot\nabla\rho &= -\div R \\
\div u &=0
\end{split}
\end{equation}
and satisfies 
\begin{equation}\label{eq:sub-cond-trans}
		\frac{1}{p} \|\rho(t)\|_{L^p}^p + \frac{1}{p'} \|u(t)\|_{L^{p'}}^{p'} + M\|R(t)\|_{L^1} < e(t) \quad\textrm{ for all }t\in[0,T].
\end{equation}
\end{subequations}

We call $\X_0$ the space of all subsolutions and equip it with the metric $d_{\X}$, given by
\begin{equation*} 
d_{\X} \left( (\rho,u), (\rho',u') \right) \coloneqq \sup_{t\in[0,T]} \Big( \|\rho(t)-\rho'(t)\|_{L^1}  + \|u(t)-u'(t)\|_{W^{1,\tilde{p}}} \Big).
\end{equation*}
Now define $\X\subset C\left( [0,T], L^1\times W^{1,\tilde{p}}  \right)$ to be the completion of $\X_0$ in the topology induced by $d_{\X}$. It is clear that $\X$ is a complete metric space and thus a Baire space and it has infinite cardinality. We will write $\tox$ for convergence in~$d_{\X}$.
	
\begin{remark}\label{remark:bddX}
	Observe that, because tuples in $\X_0$ take values in a bounded subset of $L^p\times L^{p'}$ by \eqref{eq:sub-cond-trans}, so do elements of $\X$. Indeed, this follows from Fatou's lemma. Since $p,p'\in (1,\infty)$, it then follows that $(\rho_k,u_k)\tox(\rho,u)$ implies $\rho_k(t) \rightharpoonup \rho(t)$ weakly in $L^p$ and $u_k (t) \rightharpoonup u(t)$ weakly in $L^{p'}$ for every $t$. 
\end{remark}

\subsubsection{Energy functional and statement of the main result}
	
We define a functional on $\X$ which measures the maximum distance of the energy of a tuple of functions from the given energy profile.
\begin{subequations}
	\begin{align}
	E(\rho,u)(t)&\coloneqq \frac{1}{p} \|\rho(t)\|_{L^p}^p +\frac{1}{p'} \|u(t)\|_{L^{p'}}^{p'} \label{eq:def-E} \\
	i(\rho,u)(t) &\coloneqq \Bigl( e(t)-E(\rho,u)(t)\Bigr)
	\label{eq:def-i} \\
	I(\rho,u) &\coloneqq \max_{t\in[0,T]} i(\rho,u)(t)
	\label{eq:def-I}
	\end{align}
\end{subequations}
By \eqref{eq:sub-cond-trans} and the compactness of the time interval $i$ and $I$ take only strictly positive values on $\X_0$. Using Remark \ref{remark:bddX} we deduce that $I\geq 0$ on $\X$.
	
We can now formulate the main theorem.
\begin{theorem}\label{thm:main}
	The set of functions $(\rho,u)\in \X$ which
	\begin{enumerate}
		\item are strongly continuous in time, i.e.~$(\rho,u)\in C([0,T];L^p\times L^{p'})$,
		\item solve the transport equation \eqref{eq:transport} in the sense of distributions,
		\item have energy profile $e$, i.e. $E(\rho,u)(t)=e(t)$ for all $t\in[0,T]$
	\end{enumerate}
	is residual in X.
\end{theorem}
The theorem is proved by concatenating three statements about the space $\X$ and the functional $I$.
	
	\begin{claim}
		The functional $I$ is a Baire-1-map on $\X$.
		\label{clm:I-Baire-1}
	\end{claim}
	
	\begin{claim}
		If $(\rho,u)\in \X$ is a point of continuity of $I$ then $I(\rho,u)=0$.
		\label{clm:continuity-to-energy}
	\end{claim}
	
	\begin{claim}
		If $(\rho,u)\in \X$ such that $I(\rho,u) =0$, then $(\rho,u)$ is strongly continuous in $L^p\times L^{p'}$ and a solution of \eqref{eq:transport}.
		\label{clm:energy-to-solution}
	\end{claim}
	Since the points of continuity of a Baire-1-map form a residual set, Claims~\ref{clm:I-Baire-1}, \ref{clm:continuity-to-energy} and \ref{clm:energy-to-solution} indeed prove Theorem~\ref{thm:main}. \qed
	
	\subsubsection{Proof of the main result}
	
\begin{proof}[Proof of Claim~\ref{clm:I-Baire-1}]
	We will show that $I$ is upper semicontinuous. Since moreover it takes values in the bounded interval $[0,\sup_t e]$, by Proposition~11 in Chapter~IX, section~2.7 of~\cite{Bourbaki_Gen_Top2} this implies the Baire-1 property.
		
	Let us assume, towards a contradiction, that there is a sequence
	\begin{equation*}
	(\rho_n,u_n) \tox (\rho,v) \text{ and } \lim_{n\to\infty} I(\rho_n,u_n)> I(\rho,u). 
	\end{equation*}
	Then there exists a sequence of times $t_n\in [0,T]$ such that
	\begin{equation}\label{e:usc}
	\lim_{n\to\infty} i(\rho_n,u_n)(t_n) > I(\rho,u) 
	\end{equation}
	and without loss of generality $t_n\to t_0\in [0,T]$. 
	On the other hand, recall that $\X\subset C([0,T],L^1\times W^{1,\tilde p})$ with the uniform bound $E(\rho,u)(t)\leq \sup_{t\in [0,T]}e(t)$. In particular, using the Banach-Alaoglu theorem and after extracting a subsequence (not relabeled), we deduce that  
\[ 
\rho_n(t_n) \rightharpoonup \rho(t_0) \text{ in $L^p$ and } v_n(t_n) \rightharpoonup v(t_0)\textrm{ in }L^{p'} 
\]
as $n\to\infty$. Since the $L^p$ and $L^{p'}$ norms are lower semicontinuous, we deduce 
		$ \liminf_{n\to\infty} \|\rho_n(t_n)\|_{L^p} \ge \|\rho(t_0)\|_{L^p} $
		and $ \liminf_{n\to\infty} \|u_n(t_n)\|_{L^{p'}} \ge \|u(t_0)\|_{L^{p'}}$. By the definition of $i$ and $I$ this implies
\[ 
\limsup_{n\to\infty} i(\rho_n,v_n)(t_n) \le i(\rho,v)(t_0) \le I(\rho,v), 
\]
		contradicting \eqref{e:usc}. 
\end{proof}
	
\begin{proof}[Proof of Claim~\ref{clm:continuity-to-energy}]
This is the heart of the argument and follows easily from the following perturbation property.

\begin{proposition}[Quantitative perturbation property]\label{prop:lifting}
		There exists $\alpha<1$ with the following property.
		For any $(\rho,u)\in \X_0$ there exists a sequence $(\rho_k,u_k)\in \X_0$, $k=1,2,\dots$ with
	\begin{enumerate}[(i)]
				\item $(\rho_k,u_k) \tox (\rho,u)$
				\item $I(\rho_k,u_k) \leq \alpha I(\rho,u)$.
	\end{enumerate}
			
\end{proposition}
The proof of Proposition \ref{prop:lifting} will be given in the next sections. For the moment let us show how the claim follows.		
		
Assume towards a contradiction that $(\rho,u)\in \X$ is a point of continuity of $I$ and $I(\rho,u)\ne0$. Since $I\geq 0$ on $\X$, we may thus assume $I(\rho,u)>0$. Let us also fix $\beta\in(\alpha,1)$.
		
By the continuity assumption and density of $\X_0\subset\X$ we can choose a sequence of smooth subsolutions $(\rho_n,u_n)\in \X_0$ with $(\rho_n,u_n)\tox (\rho,u)$ and $I(\rho_n,u_n)\to I(\rho,u)$. By renumbering the sequence, if necessary, we may assume without loss of generality that $I(\rho_n,u_n)\leq \frac{\beta}{\alpha} I(\rho,u)$. We apply Proposition \ref{prop:lifting} to each $(\rho_n,u_n)$ to form sequences $(\rho_{n,k},u_{n,k})\in \X_0$ with $(\rho_{n,k},u_{n,k}) \tox (\rho_n,v_n)$ as $k\to\infty$ and 
\begin{equation*}
    I(\rho_{n,k},u_{n,k})\leq \alpha I(\rho_n,u_n)\leq \beta I(\rho,u).
\end{equation*}
On the other hand the diagonal sequence satisfies $(\rho_{n,n},u_{n,n})\tox (\rho,u)$. Since $\beta<1$ we obtain a contradiction to the assumption that $I$ is continuous as $(\rho,u)\in \X$.
\end{proof}
	
\begin{proof}[Proof of Claim~\ref{clm:energy-to-solution}]
Let $(\rho,u)\in\X$ with $I(\rho,u)=0$, and let $(\rho_k,u_k)\in\X_0$ with $(\rho_k,u_k)\tox(\rho,u)$ as $k\to\infty$.
By definition $E(\rho_k,u_k)(t)<e(t)=E(\rho,u)(t)$ for all $t\in [0,T]$.
On the other hand recall from Remark \ref{remark:bddX} that $\rho_k(t) \rightharpoonup \rho(t)$ in $L^p$ and $u_k(t) \rightharpoonup u(t)$ in $L^{p'}$ for every~$t$. Therefore, by convexity of the norm
\begin{equation*}
    E(\rho,u)(t)\leq \liminf_{k\to\infty}E(\rho_k,u_k)(t)\leq \limsup_{k\to\infty}E(\rho_k,u_k)(t)\leq E(\rho,u)(t),
\end{equation*}
consequently $E(\rho_k,u_k)(t)\to E(\rho,u)(t)$ for all $t$. Since for $p,p'\in (1,\infty)$ the norms are uniformly convex, we deduce $\rho_k(t)\to \rho(t)$ strongly in $L^p$ and $u_k(t) \to u(t)$ strongly in $L^{p'}$.

Concerning equation \eqref{eq:trans-def} we first observe that, using H\"older's inequality,
the product $(\rho_ku_k)(t)$ converges strongly in $L^1$ to $(\rho u)(t)$. Furthermore, for every $k$ there exists $R_k$ such that the triple $(\rho_k,u_k,R_k)$ solves \eqref{eq:trans-def} and \eqref{eq:sub-cond-trans} is satisfied. But then 
\begin{equation*} 
\limsup_{k} M\|R_k(t)\|_{L^1} \leq e(t)- \lim_k E(\rho_k,u_k)(t)=0, 
\end{equation*}
so that $R_k(t)\to 0$ strongly in $L^1$. 
Finally note that, as functions of time, $\rho_k(t)$, $(\rho_ku_k)(t)$ and $R_k(t)$ converge pointwise (in $t$) and are uniformly bounded, hence by dominated convergence	
\begin{equation*} 
\|\rho_k-\rho\|_{L^1_t(L^p_x)}, \|\rho_ku_k-\rho u\|_{L^1_t(L^1_x)}, \|R_k\|_{L^1_t(L^1_x)} \to 0. 
\end{equation*}
Consequently $(\rho,u)$ is a solution of the transport equation in the sense of distributions.
		
		The strong continuity in time of the solution can be seen by a very similar argument, using the smoothness of $e(t)$ and the lower semicontinuity of both $t\mapsto \|\rho(t)\|_{L^p}$ and $t\mapsto \|v(t)\|_{L^p}$.
	\end{proof}

	\subsection{The Main Proposition}\label{sec:mainprop}
	The key statement on which the whole approach rests, is a perturbation property as it has been introduced in \cite{kirchheim03}. In the context of the transport equation such statements appeared in \cite{MR4029736,MR3884855} and \cite{ModenaSattig2020}. Indeed, let us recall the main proposition, Proposition 2.1 from \cite{ModenaSattig2020}, restated here in slightly modified form to suit our purposes. 

	\begin{proposition}\label{p:MainProp}
		There is a constant $M\geq 1$ such that the following holds. Let $p\in [1,\infty)$ and $\tilde p\in [1,\infty)$ so that 
		\begin{equation*}
		    \frac{1}{p}+\frac{1}{\tilde p}>1+\frac{1}{d}.
		\end{equation*}
		Then for any $\delta>0$ and any smooth solution $(\rho_0,u_0,R_0)$ of the continuity-defect equation \eqref{eq:trans-def} there is another smooth solution $(\rho_1,u_1,R_1)$ with estimates
		\begin{subequations}
		\begin{align}
		\frac{1}{p}\|(\rho_1-\rho_0)(t)\|_{L^p}^p + \frac{1}{p'} \|(u_1-u_0)(t)\|_{L^{p'}}^{p'} &\leq M\|R_0(t)\|_{L^1}\,,
		\label{eq:inductive-energy} \\
		\|(u_1-u_0)(t)\|_{W^{1,\tilde{p}}}+\|R_1(t)\|_{L^1} &\leq\delta\,, \label{eq:inductive-W1p} \\
		\|(\rho_1-\rho)(t)\|_{L^1}+\|(u_1-u_0)(t)\|_{L^1} &\leq \delta\,, \label{eq:inductive-L1}\\
		\|(\rho_1-\rho_0)(u_1-u_0)(t)\|_{L^1} &\geq  \|R_0(t)\|_{L^1} - \delta \label{eq:inductive-R}
		\end{align}
		\end{subequations}
		for all $t\in[0,T]$.
		\label{prop:main}
	\end{proposition}
	\begin{remark}\label{rem:proof}
	The proposition as stated here differs slightly from Proposition 2.1 in \cite{ModenaSattig2020}. Therefore we provide a proof sketch, still heavily based on \cite{ModenaSattig2020}.
	\end{remark}
	
	\begin{remark}\label{rem:C-def}
		The constant in the functional setup \eqref{eq:sub-cond-trans} is fixed as $M\geq 1$ in the statement of the proposition above.
	\end{remark}
	
	\begin{proof}[Proof sketch of Proposition \ref{p:MainProp}]
	We follow closely the proof of \cite[Proposition 2.1]{ModenaSattig2020} with the choice $\eta=1$ therein.
	As a start, observe that \eqref{eq:inductive-energy}-\eqref{eq:inductive-W1p} follow directly from the statement of Proposition 2.1, estimates (2.3a)-(2.3d), in \cite{ModenaSattig2020}. Therefore, in the following we focus on the additional estimates \eqref{eq:inductive-L1}-\eqref{eq:inductive-R}.
	
	 Recall from \cite[Section 4.2]{ModenaSattig2020} the definition
	 \begin{align*}
	     \rho_1(t,x)&=\rho_0(t,x)+\vartheta(t,x)+\vartheta_c(t)+q(t,x)+q_c(t)\\
	     u_1(t,x)&=u_0(t,x)+w(t,x)+w_c(t,x),
	 \end{align*}
	 where the principal perturbations are given in terms of \emph{Mikado densities} $\Theta^j$ and \emph{Mikado fields} $W^j$ as 
	 \begin{equation*}
	     \vartheta(t,x)=\sum_ja_j(t,x)\Theta^j(t,x)\,\quad 
	     w(t,x)=\sum_jb_j(t,x)W^{j}(t,x)
	 \end{equation*}
	 and $\vartheta_c,w_c$ as well as $q,q_c$ are corrector terms. These all depend on large oscillation and concentration parameters $\lambda\gg 1$, $\mu=\lambda^{\alpha}$, $\omega=\lambda^\beta$, $\nu=\lambda^{\gamma}$, see \cite[Section 6]{ModenaSattig2020}. The corrector terms are estimated in \cite[Lemma 4.6, Lemma 4.7, Lemma 4.11]{ModenaSattig2020}. With an appropriate choice of parameters in \cite[Section 6.1]{ModenaSattig2020} these estimates are shown to imply
	 \begin{equation}\label{eq:correctors}
	    \sup_t\left(\|q(t)\|_{L^p}+\|w_c(t)\|_{L^{p'}}+|\vartheta_c(t)|+|q_c(t)|\right)=o(1)\textrm{ as }\lambda\to\infty. 
	 \end{equation}
	 Concerning the principal perturbations, the $L^1$ bounds in \eqref{eq:inductive-L1} in the form 
	 \begin{equation*}
	     \sup_t\left(\|(u_1-u_0)(t)\|_{L^1}+\|(\rho_1-\rho)(t)\|_{L^1}\right)=o(1)\textrm{ as }\lambda\to\infty
	 \end{equation*}
	 follow immediately from \cite[(4.13b) and (4.16b)]{ModenaSattig2020}. 
	 
	 We now turn to \eqref{eq:inductive-R}. Using \eqref{eq:correctors} it suffices to estimate $\|\vartheta w\|_{L^1}-\|R_0\|_{L^1}$. Recall from \cite[Section 4.2]{ModenaSattig2020} that
	 \begin{align*}
	     \vartheta(t,x) w(t,x)&=\sum_{j=1}^da_j(t,x)b_j(t,x)\Theta^j(t,x)W^j(t,x)\\
	     &=\sum_{j=1}^d\chi_j^2(t,x)R_0^j(t,x)F_{\mu}^j(\lambda(x-\omega te_j))[\psi^j(\nu x)]^2e_j,
	 \end{align*}
	 where $\chi_j$ are smooth cut-off functions with 
	 \begin{equation}\label{eq:cutoff}
	     \chi_j(t,x)=\begin{cases}0&\textrm{ if }|R_0^j(t,x)|\leq \varepsilon,\\ 1&\textrm{ if }|R_0^j(t,x)|\geq 2\varepsilon\end{cases}
	 \end{equation}
	 for some $\varepsilon=C\delta$ (set to be $\tfrac{1}{4d}\delta$ in \cite[Section 4.2]{ModenaSattig2020}). Moreover
	 \begin{equation*}
	     F_{\mu}^j(x)=\varphi^j_{\mu}(x)\tilde\varphi^j_{\mu}(x),
	 \end{equation*}
	 and $\{e_j:\,j=1,\dots,d\}$ is the standard orthonormal basis of $\R^d$ so that $R_0=\sum_{j=1}^dR_0^je_j$.
	 The periodic functions $\varphi^j_{\mu}, \tilde\varphi^j_{\mu}$ and $\psi^j$ are defined in \cite[Lemma 4.3 and (4.8)]{ModenaSattig2020}, and in particular it follows that $F_\mu\geq 0$,
	 $F^j_\mu(\lambda(\cdot-\omega te_j))$ and $F^{j'}_\mu(\lambda(\cdot-\omega te_{j'}))$ have disjoint (spatial) support for every $t$ for $j\neq j'$, and
	 \begin{equation}\label{eq:averages}
	 \int F_\mu(x)\,dx=\int [\psi^j(x)]^2\,dx=1.
	 \end{equation}
	 Let us write $w=\sum_jw_je_j$. Then $\vartheta w=\sum_j\vartheta w_j e_j$ with $\textrm{supp }\vartheta w_j\cap \textrm{supp }\vartheta w_{j'}=\emptyset$. Consequently $|\vartheta w|=\sum_j|\vartheta w_j|$. Moreover
	 \begin{align*}
	     |\vartheta(t,x) w_j(t,x)|=&\chi_j^2(t,x)|R_0^j(t,x)|F_{\mu}(\lambda(x-\omega te_j))[\psi^j(\nu x)]^2\\
	     =&|R_0^j(t,x)|+(\chi_j^2(t,x)-1)|R^j_0(t,x)|+\\
	     &+\chi^2_j(t,x)|R_0^j(t,x)|(F_\mu(\lambda(x-\omega te_j)-1)+\\
	     &+\chi^2_j(t,x)|R_0^j(t,x)|F_\mu(\lambda(x-\omega te_j)([\psi^j(\nu x)]^2-1)\\
	     =&|R_0^j(t,x)|+(I)+(II)+(III).
	 \end{align*}
	It follows from \eqref{eq:cutoff} that $(I)=O(\varepsilon)$ uniformly in space. As for (II) and (III), because of \eqref{eq:averages} we may apply \cite[Lemma 2.6]{ModenaSzekelyhidi18} to obtain
	\begin{align*}
	    \int_{\Td} (II)&\leq \frac{C}{\lambda}\|\chi_j^2|R_0^j|\|_{C^1}=O(\lambda^{-1})\\
	    \int_{\Td} (III)&\leq \frac{C}{\nu}\|\chi_j^2|R_0^j|F_{\mu}(\lambda\cdot)\|_{C^1}=O\left(\frac{\lambda\mu}{\nu}\right).
	\end{align*}
	Again using the choice of parameters in \cite[Section 6.1]{ModenaSattig2020}, and in particular (6.1b) therein, we deduce
	\begin{equation*}
	    \int_{\Td} |\vartheta w_j|= \int_{\Td} |R^j_0|+O(\varepsilon)+o(1)\quad\textrm{ as }\lambda\to\infty.
	\end{equation*}
	By choosing appropriately $\varepsilon=O(\delta)$ and $\lambda\gg 1$ we deduce
	\begin{equation*}
	    \|\vartheta w\|_{L^1} =\sum_{j=1}^d\int_{\Td} |\vartheta w^j|\geq \sum_{j=1}^d\int_{\Td} |R_0^j|\,dx -\delta \geq \|R_0\|_{L^1}-\delta
	\end{equation*}
	from which \eqref{eq:inductive-R} follows.
	\end{proof}

\begin{corollary}\label{c:additional}
In the setting of Proposition \ref{p:MainProp} we can ensure, in addition to \eqref{eq:inductive-energy}-\eqref{eq:inductive-R}, the estimate
\begin{equation}\label{eq:inductive-E}
\begin{split}
\Bigl|\|\rho_1(t)\|_{L^p}^p-\|\rho_0(t)\|_{L^p}^p-\|(\rho_1-\rho_0)(t)\|_{L^p}^p\Bigr|\leq \delta\\
\Bigl|\|u_1(t)\|_{L^{p'}}^{p'}-\|u_0(t)\|_{L^{p'}}^{p'}-\|(u_1-u_0)(t)\|_{L^{p'}}^{p'}\Bigr|\leq \delta\\
\end{split}
\end{equation}
for any $t\in [0,T]$.
\end{corollary}

The proof of the corollary is based on
the following lemma, which can be seen as a quantitative version of a famous convergence result of Brezis and Lieb \cite{BrezisLieb}, see also \cite[Theorem 1.9]{LiebLoss}.
		\begin{lemma}
			Let $p\in[1,2)$, then there is a constant $C_p<\infty$ such that
			\begin{equation}
			\left| \|f+g\|_{L^p}^p - \|f\|_{L^p}^p - \|g\|_{L^p}^p \right| \le C_p \left( \|f\|_{L^{\frac{1}{2-p}}} \|g\|_{L^1}^{p-1} + \|f\|_{L^\infty}^{p-1} \|g\|_{L^1} \right)
			\label{eq:Lp-sum-decoupling-low}
			\end{equation}
			for every (scalar or vectorial) $f,g$. 
			If $p\in[2,\infty)$, then
			\begin{equation}
			\left| \|f+g\|^p_{L^p}-\|f\|^p_{L^p}-\|g\|^p_{L^p} \right| \le C_p \left( \|f\|_{L^\infty} \|g\|_{L^{p-1}}^{p-1} + \|f\|_{L^\infty}^{p-1} \|g\|_{L^1} \right)
			\label{eq:Lp-sum-decoupling-high}
			\end{equation}
		\end{lemma}
		\begin{proof}
			Both inequalities follow from the following pointwise inequality
			\[ \left| |a+b|^p-|a|^p-|b|^p \right| \le C_p \left( |a||b|^{p-1} + |a|^{p-1}|b| \right) \]
			by straightforward application of Hölder's inequality. Using homogeneity and symmetry the pointwise version is a consequence of
			\[ \sup_{\xi\in B_1(0) \setminus\{0\}} \frac{ \left| |e_1+\xi|^p-1-|\xi|^p \right| }{|\xi|+|\xi|^{p-1}} = C_p < \infty \]
			which can be verified easily.
		\end{proof}

\begin{proof}[Proof of Corollary \ref{c:additional}]
Since all estimates are uniform in $t$, we will supress dependence on $t$ in what follows. 
Let $A=\Bigl|\|\rho_1\|_{L^p}^p-\|\rho_0\|_{L^p}^p-\|\rho_1-\rho_0\|_{L^p}^p\Bigr|$. We will show that the estimates \eqref{eq:inductive-energy} and \eqref{eq:inductive-L1} imply $|A|=o(1)$ as $\delta\to 0$. Then, since $\delta>0$ is arbitrary, by fixing a smaller value if necessary, we deduce the statement of the corollary.

If $p<2$ apply \eqref{eq:Lp-sum-decoupling-low} to estimate
\begin{align*}
		|A| &\le C_p \left( \|\rho_0\|_{L^{\frac{1}{2-p}}} \|\rho_1-\rho_0\|_{L^1}^{p-1} + \|\rho_0\|_{L^\infty}^{p-1} \|\rho_1-\rho_0\|_{L^1} \right) \\
		&\lesssim \delta^{p-1} + \delta
\end{align*}
		where the implied constant depends on the given smooth density $\rho_0$. 
For $p\ge 2$ we apply \eqref{eq:Lp-sum-decoupling-high} and interpolate \eqref{eq:inductive-energy} and \eqref{eq:inductive-L1} to estimate the $L^{p-1}$-norm of the perturbation:
\begin{align*}
		|A| &\le C_p \left( \|\rho_0\|_{L^\infty} \|\rho_1-\rho_0\|_{L^{p-1}}^{p-1} + \|\rho_0(t)\|_{L^\infty}^{p-1} \|\rho_1-\rho_0\|_{L^1} \right) \\
		&\le C_p \left(  \|\rho_0\|_{L^\infty} \|\rho_1-\rho_0\|_{L^1}^{\frac{1}{p-1}} \|\rho_1-\rho_0\|_{L^p}^{\frac{p(p-2)}{p-1}} +  \|\rho_0\|_{L^\infty}^{p-1} \|\rho_1-\rho_0\|_{L^1} \right) \\
		&\lesssim \delta^{\frac{1}{p-1}} +\delta.
\end{align*}
In both cases we see that $|A|=o(1)$ as $\delta\to 0$. Since $\delta>0$ was arbitrary in Proposition \ref{p:MainProp}, we deduce Corollary \ref{c:additional} by choosing a smaller $\delta>0$ if necessary.
\end{proof}

\subsection{Proof of the quantitative perturbation property}
\label{sec:perturbation}

\begin{proof}[Proof of Proposition~\ref{prop:lifting}] Let $(\rho,u)\in \X_0$ and let $R$ be a smooth time-dependent vector field such that the triple $(\rho,u,R)$ solves \eqref{eq:trans-def} with \eqref{eq:sub-cond-trans}. Denote 
\begin{equation*}
    E(t):=E(\rho,u)(t),\quad i(t):=i(\rho,u)(t)=e(t)-E(t).
\end{equation*}
Since $i$ is continuous on the compact interval $[0,T]$ with $i(t)>M\|R(t)\|_{L^1}$ (from \eqref{eq:sub-cond-trans}), there exists $\delta>0$ so that, for all $t\in [0,T]$,
\begin{equation}\label{e:iLB}
    i(t)\geq M\|R(t)\|_{L^1}+42M\delta.
\end{equation}

We apply Proposition \ref{prop:main} to $(\rho,u,R)$ and obtain $(\tilde\rho,\tilde u,\tilde R)$ with
\begin{subequations}
\begin{align}
    \|\tilde R(t)\|_{L^1}&\leq \delta,\label{e:step1-R}\\
    d_{\X}((\tilde\rho,\tilde u),(\rho,u))&\leq 2\delta,\label{e:step1-d}\\
    \tilde E(t)&\leq E(t)+M\|R(t)\|_{L^1}+\delta,\label{e:step1-UB}\\
    \tilde E(t)&\geq E(t)+\|R(t)\|_{L^1}-2\delta,\label{e:step1-LB}
\end{align}
\end{subequations}
where $\tilde E(t):=E(\tilde \rho,\tilde u)(t)$. Indeed, \eqref{e:step1-R} and \eqref{e:step1-d} follow directly from \eqref{eq:inductive-W1p} and \eqref{eq:inductive-L1}. To see \eqref{e:step1-UB} we use Corollary \ref{c:additional} and \eqref{eq:inductive-energy} to obtain
\begin{align*}
    \tilde E(t)&\leq E(t)+\frac{1}{p}\|(\tilde\rho-\rho)(t)\|_{L^p}^p+\frac{1}{p'}\|(\tilde u-u)(t)\|_{L^{p'}}^{p'}+\delta\\
    &\leq E(t)+M\|R(t)\|_{L^1}+\delta.
\end{align*}
Similarly, to see \eqref{e:step1-LB} we use Corollary \ref{c:additional}, Young's inequality and \eqref{eq:inductive-R} to obtain
\begin{align*}
    \tilde E(t)&\geq E(t)+\frac{1}{p}\|(\tilde \rho-\rho)(t)\|_{L^p}^p+\frac{1}{p'}\|(\tilde u-u)(t)\|_{L^{p'}}^{p'}-\delta\\
    &\geq E(t)+\|(\tilde \rho-\rho)(t)(\tilde u-u)(t)\|_{L^1}-\delta\\
    &\geq E(t)+\|R(t)\|_{L^1}-2\delta.
\end{align*}

Next, for a smooth function $\chi:[0,T]\to\R$ still to be fixed, we define 
\begin{equation*}
\bar R(t,x)=\tilde R(t,x)+\frac{1}{|\Td|}\chi(t)v,
\end{equation*}
where $v\in\R^d$ is an arbitrary but fixed unit vector. From \eqref{e:step1-R} we deduce 
\begin{equation*}
\chi(t)-\delta\leq \|\bar R(t)\|_{L^1}\leq \chi(t)+\delta.
\end{equation*}
We apply Proposition \ref{prop:main} to $(\tilde\rho,\tilde u,\bar R)$ and obtain $(\rho_1,u_1,R_1)$ with
\begin{subequations}
\begin{align}
    \|R_1(t)\|_{L^1}&\leq \delta,\label{e:step2-R}\\
    d_{\X}((\rho_1,u_1),(\tilde \rho,\tilde u))&\leq 2\delta,\label{e:step2-d}\\
    E_1(t)&\leq \tilde E(t)+M\|\bar R(t)\|_{L^1}+\delta\nonumber\\
    &\leq  \tilde E(t)+M\chi(t)+(M+1)\delta,\nonumber\\
    &\leq E(t)+M(\chi(t)+\|R(t)\|_{L^1})+(M+2)\delta\label{e:step2-UB}\\
    E_1(t)&\geq \tilde E(t)+\|\bar R(t)\|_{L^1}-2\delta\nonumber\\
    &\geq \tilde E(t)+\chi(t)-3\delta,\nonumber\\
    &\geq E(t)+(\chi(t)+\|R(t)\|_{L^1})-5\delta,\label{e:step2-LB}
\end{align}
\end{subequations}
where $E_1(t):=E(\rho_1,u_1)(t)$. 

Now we are ready to fix a smooth $\chi:[0,T]\to[0,\infty)$ so that, for all $t\in [0,T]$
\begin{align*}
    \frac{1}{M}i(t)-\|R(t)\|_{L^1}-16\delta< \chi(t)< \frac{1}{M}i(t)-\|R(t)\|_{L^1}-8\delta.
\end{align*}
To see that such $\chi$ exists, it suffices to observe that the right hand side is bounded below by $34\delta$ because of \eqref{e:iLB}.

With this choice of $\chi$ we can estimate
\begin{align*}
    E_1(t)+M\|R_1(t)\|_{L^1}&\leq E_1(t)+M\delta\\
    &\leq E(t)+M(\chi(t)+\|R(t)\|_{L^1})+2(M+1)\delta\\
    &\leq e(t)-4M\delta,
\end{align*}
so that $(\rho_1,u_1)\in\X_0$. Furthermore, 
\begin{align*}
    E_1(t)&\geq E(t)+(\chi(t)+\|R(t)\|_{L^1})-5\delta\\
    &\geq E(t)+\frac{1}{M}(e(t)-E(t))-21\delta\\
    &\overset{\textrm{\eqref{e:iLB}}}{\geq} E(t)+\frac{1}{2M}(e(t)-E(t))
\end{align*}
so that
\begin{equation*}
    i(\rho_1,u_1)(t)=e(t)-E_1(t)\leq \left(1-\frac{1}{2M}\right)(e(t)-E(t))=\left(1-\frac{1}{2M}\right)i(\rho,u)(t).
\end{equation*}
In particular
\begin{equation}\label{e:improvement}
I(\rho_1,u_1)\leq \left(1-\frac{1}{2M}\right)I(\rho,u).
\end{equation}
Recall also $d_{\X}((\rho_1,u_1),(\rho,u))\leq 4\delta$.
Then, we define a sequence $(\rho_k,u_k)$, $k=1,2,\dots$ by the above construction with $\delta_k=2^{1-k}\delta$. The estimate \eqref{e:improvement} continues to hold for every $(\rho_k,u_k,R_k)$, whereas $(\rho_k,u_k)\tox (\rho,u)$. This concludes the proof of the proposition. 
\end{proof}
	
\section{The Navier-Stokes equation}
\label{sec:ns}

In this section we show how the same argument as above and a construction contained in \cite{BurczakModenaSzekelyhidi20} are sufficient to show the genericity of distributional solutions to the Navier-Stokes equations
	\begin{equation}
	\label{eq:ns}
		\begin{aligned}
			\partial_t v + v\cdot\nabla v +\nabla p &= \Delta v \\
			 \div v &= 0
		\end{aligned}
	\end{equation}
with Sobolev regularity. For simplicity we again treat the case of periodic boundary conditions, so that our spatial domain is the $d$-dimensional flat torus $\Td$, with $d\geq 3$. 

\subsection{Functional setup}\label{sec:functional-ns}
\subsubsection{The space of subsolutions and its metric}
	
Let $e:[0,T]\to \R_{>0}$ be a strictly positive smooth function. A subsolution is smooth vector field $v:[0,T]\times \Td\to \R^d$ such that there exists a smooth traceless symmetric tensor field $\re:[0,T]\times \Td\to \R^{d\times d}$ such that the triple $(v,\re)$ solves the Navier-Stokes-Reynolds system
\begin{subequations}
\begin{equation}\label{eq:NS-def}
\begin{split}
\partial_t v + v\cdot\nabla v +\nabla p&= \Delta v+\div \re \\
\div v &=0
\end{split}
\end{equation}
and satisfies 
\begin{equation}\label{eq:sub-cond-ns}
		\frac{1}{2} \|v(t)\|_{L^2}^2 + d\|\re(t)\|_{L^1} < e(t) \quad\textrm{ for all }t\in[0,T].
\end{equation}
\end{subequations}

\begin{remark}
    Recall that $(v,\re)$ is a distributional solution of \eqref{eq:NS-def} if $v$ is divergence-free and for any divergence-free test function $\Phi\in C^{\infty}(\Td\times (0,T);\R^d)$ we have
    $\int_0^T\int_{\Td} v\cdot\partial_t\Phi+(v\otimes v-\re):\nabla\Phi\,dx\,dt=0$. In particular, for any distributional solution the pressure is determined uniquely from the equation $-\Delta p=\div\div(v\otimes v-\re)$ and the condition $\int_{\Td}p\,dx=0$.
\end{remark}

We call $\X_0$ the space of all subsolutions and equip it with the metric $d_{\X}$, given by
\begin{equation*} 
d_{\X} \left( v, v' \right) \coloneqq \sup_{t\in[0,T]} \|v(t)-v'(t)\|_{W^{1,\tilde{p}}}.
\end{equation*}
Now define $\X\subset C\left( [0,T], W^{1,\tilde{p}}  \right)$ to be the completion of $\X_0$ in the topology induced by $d_{\X}$. Then $\X$ is a complete metric space and thus a Baire space and it has infinite cardinality. We will write $\tox$ for convergence in~$d_{\X}$.
	
\begin{remark}\label{remark:bddX-ns}
	Observe that, because elements of $\X_0$ take values in a bounded subset of $L^2$ by \eqref{eq:sub-cond-ns}, so do elements of $\X$. Indeed, this follows from Fatou's lemma. It then follows that $v_k\tox v$ implies $v_k (t) \rightharpoonup v(t)$ weakly in $L^{2}$ for every $t$. 
\end{remark}

\subsubsection{Energy functional and statement of the main result}
	
We define a functional on $\X$ which measures the maximum distance of the energy of subsolutions from the given energy profile.
\begin{subequations}
	\begin{align}
	E(v)(t)&\coloneqq \frac{1}{2} \|v(t)\|_{L^2}^2  \label{eq:def-E-ns} \\
	i(v)(t) &\coloneqq \Bigl( e(t)-E(v)(t)\Bigr)
	\label{eq:def-i-ns} \\
	I(v) &\coloneqq \max_{t\in[0,T]} i(v)(t)
	\label{eq:def-I-ns}
	\end{align}
\end{subequations}
By \eqref{eq:sub-cond-ns} and the compactness of the time interval $i$ and $I$ take only strictly positive values on $\X_0$. Using Remark \ref{remark:bddX} we deduce that $I\geq 0$ on $\X$.
	
We can now formulate the main theorem of this section.
\begin{theorem}\label{thm:main-ns}
	The set of functions $v\in \X$ which
	\begin{enumerate}
		\item are strongly continuous in time, i.e.~$v\in C([0,T];L^2)$,
		\item solve the Navier-Stokes equations \eqref{eq:ns} in the sense of distributions,
		\item have energy profile $e$, i.e. $E(v)(t)=e(t)$ for all $t\in[0,T]$
	\end{enumerate}
	is residual in X.
\end{theorem}

\begin{remark}
	    In \cite{ColomboDeRosaSorella21} the authors proved genericity of irregular solutions to the Navier-Stokes equations in the space of all $L_t^\infty L_x^2$ solutions by an explicit approximation argument. In contrast, our theorem provides genericity of solutions in a much larger space containing also subsolutions, but without a statement on partial regularity.
\end{remark}

The theorem is once again proved by concatenating three statements about the space $\X$ and the functional $I$.
	
	\begin{claim}
		The functional $I$ is a Baire-1-map on $\X$.
		\label{clm:I-Baire-1-ns}
	\end{claim}
	
	\begin{claim}
		If $v\in \X$ is a point of continuity of $I$ then $I(v)=0$.
		\label{clm:continuity-to-energy-ns}
	\end{claim}
	
	\begin{claim}
		If $v\in \X$ such that $I(v) =0$, then $v$ is strongly continuous in $L^2$ and a solution of \eqref{eq:ns}.
		\label{clm:energy-to-solution-ns}
	\end{claim}
	As in Section \ref{sec:functional}, since the points of continuity of a Baire-1-map form a residual set, Claims~\ref{clm:I-Baire-1-ns}, \ref{clm:continuity-to-energy-ns} and \ref{clm:energy-to-solution-ns} indeed prove Theorem~\ref{thm:main-ns}. \qed

The proofs of Claims \ref{clm:I-Baire-1-ns}, \ref{clm:continuity-to-energy-ns} and \ref{clm:energy-to-solution-ns} are identical to those of Claims~\ref{clm:I-Baire-1}. The key point is again the following quantitative perturbation property, the exact analogue of Proposition \ref{prop:lifting} in the current setting:
\begin{proposition}[Quantitative Perturbation Property]\label{prop:perturbation-ns}
There is a constant $\alpha<1$ such that for any $v\in\X_0$ there is a sequence $(v_k)_{k\in\N} \subset \X_0$ satisfying
\begin{equation*}
	   v_k\tox v \ \text{ and } \ I(v_k)\le \alpha I(v).
\end{equation*}
\end{proposition}

\subsection{The main proposition and proof of the quantitative perturbation property}
\label{sec:perturbation-ns}

The proof of Proposition \ref{prop:perturbation-ns} follows again from a perturbation property for weak subsolutions of the Navier-Stokes equations (c.f.~Proposition \ref{prop:main}), and is a slightly modified version of the main proposition in \cite{BurczakModenaSzekelyhidi20} (c.f.~\cite[Theorem C]{BurczakModenaSzekelyhidi20}).
	\begin{proposition}
		\label{prop:NS}
        Let $\gamma_0:[0,T]\to \R_{>0}$ smooth, strictly positive, let $\delta>0$ and $r<\frac{2d}{d+2}$. If $(v,p,\re)$ is a smooth solution to \eqref{eq:NS-def} there is another smooth solution $(v_1,p_1,\re_1)$ such that, for all $t\in[0,T]$,
        \begin{subequations}
        \begin{align}
           \|(v_1-v)(t)\|_{W^{1,r}} + 	\|\re_1(t)\|_{L^1} &< \delta\,,	\label{eq:NS-inductive-Sobolev-R} \\
			 \left| \|(v_1-v)(t)\|_{L^2}^2-2d\|\re(t)\|_{L^1} -(2\pi)^dd\gamma_0(t) \right| &< \delta\,.
		\label{eq:NS-inductive-energy}
		\end{align}
        \end{subequations}
	\end{proposition}

Let us first show that this indeed implies a quantitative perturbation property. 
\begin{proof}[Proof of Proposition \ref{prop:perturbation-ns}]
Let $v\in \X_0$ and let $\re$ be such that \eqref{eq:NS-def} and \eqref{eq:sub-cond-ns} hold. Set
	\[ 
	\gamma_0(t)= \frac{1}{(2\pi)^dd}\left(i(v)(t)-d\|\re(t)\|_{L^1})\right), 
	\]
	and note that $\gamma_0$ is strictly positive because of \eqref{eq:sub-cond-ns} and smooth. Fix a sequence $\delta_k\to 0$ of sufficiently small positive numbers and apply Proposition \ref{prop:NS} to obtain smooth solutions $(v_k,\re_k)$ of \eqref{eq:NS-def} satisfying \eqref{eq:NS-inductive-Sobolev-R}--\eqref{eq:NS-inductive-energy} with $\delta_k>0$. From \eqref{eq:NS-inductive-Sobolev-R} it is clear that $v_k\tox v$.
	
	It remains to estimate $I(v_k)$. First observe that it is sufficient to control the relative energy $\|v_k-v\|_{L^2}$ since
	\[   
	\left| \|v_k\|_{L^2}^2-\|v\|_{L^2}^2 - \|v_k-v\|_{L^2}^2 \right| = 2|\langle v, v_k \rangle_{L^2} | \le 2\|v\|_{L^\infty} \|v_k\|_{L^1} = O(\delta_k),
	\]
	where in the last step we used \eqref{eq:NS-inductive-Sobolev-R} and embedding of the Sobolev space $W^{1,r}$ into $L^1$. Therefore, by \eqref{eq:NS-inductive-energy} and the choice of $\gamma_0(t)$
 	\begin{align*}
    \|v_k(t)\|_{L^2}^2 &= \|v(t)\|_{L^2}^2 + 2d\|\re(t)\|_{L^1} + (2\pi)^dd\gamma_0(t) + O(\delta_k) \\
    &= \frac12 \|v(t)\|_{L^2}^2+e(t)   + d \|\re(t)\|_{L^1} +O(\delta_k)	
        \end{align*}
    which can be used in two directions. First to show the perturbation property:
    \begin{align*}
        I(v_k) &= \sup_t \left( e(t)-\tfrac12\|v_k(t)\|_{L^2}^2 \right) = \sup_t \tfrac12 \left( e(t)-\tfrac12\|v(t)\|_{L^2}^2 - d\|\re(t)\|_{L^1} \right) + O(\delta_k) \\
        &\le \frac12 I(v) + O(\delta_k) \le \frac34 I(v)
    \end{align*}
	if $\delta_k$ is sufficiently small (depending on $v$). On the other hand, using \eqref{eq:NS-inductive-Sobolev-R}
	\begin{align*}
	    \frac12\|v_k(t)\|_{L^2}^2 + d\|\re_k(t)\|_{L^1} = \frac12 \left( e(t) + \frac12\|v(t)\|_{L^2}^2 + d \|\re(t)\|_{L^1} \right) +O(\delta_k)
	\end{align*}
	and by \eqref{eq:sub-cond-ns} the term in brackets is strictly smaller than $2e(t)$ on the compact interval $[0,T]$, so $\delta_k$ can be chosen sufficiently small so that $(v_k,\re_k)$ satisfies \eqref{eq:sub-cond-ns} for every $t\in[0,T]$. This concludes the proof of Proposition \ref{prop:perturbation-ns}.
\end{proof}

\subsection{Proof of the main proposition}
	We show that a minor adjustment in the construction in \cite{BurczakModenaSzekelyhidi20} and an additional estimate provides a proof of Proposition~\ref{prop:NS}. More precisely, we apply the proof of \cite[Proposition 1]{BurczakModenaSzekelyhidi20} in the context of \cite[Theorem B]{BurczakModenaSzekelyhidi20} (see also \cite[Remark 11]{BurczakModenaSzekelyhidi20}), for which the choice of parameters is defined in Section 9 (and includes the 3D Navier-Stokes system, see \cite[Section 1.5]{BurczakModenaSzekelyhidi20}). More precisely, let us assume $(v,\re)$ is a smooth solution of \eqref{eq:NS-def}. We recall that the construction in \cite[Section 5]{BurczakModenaSzekelyhidi20} proceeds by defining 
$$
v_1= v+u_p + u_c,\quad \mathring{R}_1=-(R_{lin}+R_{corr}+R_{quadr}),
$$
where $u_p$ and $u_c$ are oscillatory/concentrated vectorfields based on Mikado flows, with large frequency/concentration parameters $\lambda_1,\lambda_2,\mu,\omega$, all of which are powers of some sufficiently large $\lambda\in\N$.
Since in our setting $\mathring{R}$ is assumed to be smooth, there is no need to regularize and thus the principal part of the perturbation, $u_p$, is defined in (49) using Mikado flows corresponding to the Reynolds stress term $\rho\textrm{Id}+\mathring{R}$, with $\rho=2\sqrt{\epsilon^2+|\mathring{R}(x,t)|^2}+\gamma_0(t)$ (c.f.~(44)). Although a choice for $\gamma_0(t)>0$ is specified in (45) for the purposes of control of the kinetic energy in Theorems A and B in \cite{BurczakModenaSzekelyhidi20}, for the rest of the construction it is only relevant that $\gamma_0>0$ (c.f.
~also with the proof of \cite[Theorem C]{BurczakModenaSzekelyhidi20}). In particular we deduce estimates (13b) and (13c) of \cite[Proposition 1]{BurczakModenaSzekelyhidi20}, which immediately imply \eqref{eq:NS-inductive-Sobolev-R} above. In order to verify \eqref{eq:NS-inductive-energy} we first observe that (79) and (80) from the proof of \cite[Proposition 1]{BurczakModenaSzekelyhidi20}  provide the estimate 
\begin{equation}\label{e:uc}
\|u_c\|_{L^2} \lesssim \frac{\mu
^{d/2}}{\omega} + \frac{\lambda_1\mu}{\lambda_2} \left(1+\frac{(\lambda_1\mu)^N}{\lambda_2^{N-2}} \right) = o(1) 
\end{equation}
as $\lambda\to\infty$ because of (108) and (110). Thus it remains to estimate $u_p$.
This is done in (97), (100) and (101) in \cite{BurczakModenaSzekelyhidi20}, yielding
\[ 
\int_{\Td} |u_p(x,t)|^2 =  \int_{\Td} 2d\sqrt{\epsilon^2+|\re(x,t)|^2} +d\gamma_0(t) \,dx + o(1) \]
where the $o(1)$-notation again refers to the limit $\lambda\to\infty$ and $\epsilon>0$ is an arbitrary regularization parameter. Since $|\re|\le \sqrt{\epsilon^2+|\re|^2} \le \epsilon+|\re|$ we obtain
\begin{equation}\label{e:up}
\int_{\Td}|u_p(x,t)|^2-d\gamma_0(t)-2d|\re(x,t)| \,dx= O(\epsilon)+o(1), 
\end{equation}
Since $\epsilon>0$ and $\lambda\in\N$ are arbitrary, the estimates \eqref{e:uc}-\eqref{e:up} imply \eqref{eq:NS-inductive-energy}. This concludes the proof of the proposition.\qed

	\bibliographystyle{acm}
\bibliography{baire_transport}

\begin{thebibliography}{10}

\bibitem{Bourbaki_Gen_Top2}
{\sc Bourbaki, N.}
\newblock {\em General topology. {C}hapters 5--10}.
\newblock Elements of Mathematics (Berlin). Springer-Verlag, Berlin, 1998.
\newblock Translated from the French, Reprint of the 1989 English translation.

\bibitem{BrezisLieb}
{\sc Br\'{e}zis, H., and Lieb, E.}
\newblock A relation between pointwise convergence of functions and convergence
  of functionals.
\newblock {\em Proc. Amer. Math. Soc. 88}, 3 (1983), 486--490.

\bibitem{BrueColomboDeLellis21}
{\sc Bru\'{e}, E., Colombo, M., and De~Lellis, C.}
\newblock Positive solutions of transport equations and classical nonuniqueness
  of characteristic curves.
\newblock {\em Arch. Ration. Mech. Anal. 240}, 2 (2021), 1055--1090.

\bibitem{BuckModena22}
{\sc Buck, M., and Modena, S.}
\newblock On the failure of the chain rule for the divergence of sobolev vector
  fields, 2022.

\bibitem{BDLSV19}
{\sc Buckmaster, T., de~Lellis, C., Sz\'{e}kelyhidi, Jr., L., and Vicol, V.}
\newblock Onsager's conjecture for admissible weak solutions.
\newblock {\em Comm. Pure Appl. Math. 72}, 2 (2019), 229--274.

\bibitem{BuckmasterVicol19}
{\sc Buckmaster, T., and Vicol, V.}
\newblock Nonuniqueness of weak solutions to the {N}avier-{S}tokes equation.
\newblock {\em Ann. of Math. (2) 189}, 1 (2019), 101--144.

\bibitem{BurczakModenaSzekelyhidi20}
{\sc Burczak, J., Modena, S., and Sz\'{e}kelyhidi, L.}
\newblock Non uniqueness of power-law flows.
\newblock {\em Comm. Math. Phys. 388}, 1 (2021), 199--243.

\bibitem{cellina2005}
{\sc Cellina, A.}
\newblock A view on differential inclusions.
\newblock {\em Rend. Semin. Mat. Univ. Politec. Torino 63}, 3 (2005), 197--209.

\bibitem{ColomboDeRosaSorella21}
{\sc Colombo, M., {de Rosa}, L., and Sorella, M.}
\newblock Typicality results for weak solutions of the incompressible
  navier--stokes equations, 2021.

\bibitem{dacorogna_marcellini97}
{\sc Dacorogna, B., and Marcellini, P.}
\newblock General existence theorems for {H}amilton-{J}acobi equations in the
  scalar and vectorial cases.
\newblock {\em Acta Math. 178}, 1 (1997), 1--37.

\bibitem{dacorogna_marcellini99}
{\sc Dacorogna, B., and Marcellini, P.}
\newblock {\em Implicit partial differential equations}, vol.~37 of {\em
  Progress in Nonlinear Differential Equations and their Applications}.
\newblock Birkh\"{a}user Boston, Inc., Boston, MA, 1999.

\bibitem{DeLellisSzekelyhidi09}
{\sc De~Lellis, C., and Sz\'{e}kelyhidi, Jr., L.}
\newblock The {E}uler equations as a differential inclusion.
\newblock {\em Ann. of Math. (2) 170}, 3 (2009), 1417--1436.

\bibitem{DeLellisSzekelyhidi10}
{\sc De~Lellis, C., and Sz\'{e}kelyhidi, Jr., L.}
\newblock On admissibility criteria for weak solutions of the {E}uler
  equations.
\newblock {\em Arch. Ration. Mech. Anal. 195}, 1 (2010), 225--260.

\bibitem{MR2917063}
{\sc De~Lellis, C., and Sz\'{e}kelyhidi, Jr., L.}
\newblock The {$h$}-principle and the equations of fluid dynamics.
\newblock {\em Bull. Amer. Math. Soc. (N.S.) 49}, 3 (2012), 347--375.

\bibitem{DeLellisSzekelyhidi13}
{\sc De~Lellis, C., and Sz\'{e}kelyhidi, Jr., L.}
\newblock Dissipative continuous {E}uler flows.
\newblock {\em Invent. Math. 193}, 2 (2013), 377--407.

\bibitem{delellis_szekelyhidi16}
{\sc De~Lellis, C., and Sz\'{e}kelyhidi, Jr., L.}
\newblock High dimensionality and h-principle in {PDE}.
\newblock {\em Bull. Amer. Math. Soc. (N.S.) 54}, 2 (2017), 247--282.

\bibitem{DiPernaLions89}
{\sc DiPerna, R.~J., and Lions, P.-L.}
\newblock Ordinary differential equations, transport theory and {S}obolev
  spaces.
\newblock {\em Invent. Math. 98}, 3 (1989), 511--547.

\bibitem{Ise19}
{\sc Isett, P.}
\newblock {\em H\"{o}lder continuous {E}uler flows in three dimensions with
  compact support in time}, vol.~196 of {\em Annals of Mathematics Studies}.
\newblock Princeton University Press, Princeton, NJ, 2017.

\bibitem{kirchheim03}
{\sc Kirchheim, B.}
\newblock Analysis and geometry of microstructure.
\newblock Habilitation thesis, University of Leipzig,
  https://www.mis.mpg.de/preprints/ln/lecturenote-1603.pdf, 2003.

\bibitem{Kirchheim:2002wc}
{\sc Kirchheim, B., {\v{S}}ver{\'a}k, V., and M{\"u}ller, S.}
\newblock {Studying nonlinear pde by geometry in matrix space}.
\newblock In {\em Geometric analysis and nonlinear partial differential
  equations}. Springer, Berlin, 2003, pp.~347--395.

\bibitem{LiebLoss}
{\sc Lieb, E.~H., and Loss, M.}
\newblock {\em Analysis}, vol.~14 of {\em Graduate Studies in Mathematics}.
\newblock American Mathematical Society, Providence, RI, 1997.

\bibitem{ModenaSattig2020}
{\sc Modena, S., and Sattig, G.}
\newblock Convex integration solutions to the transport equation with full
  dimensional concentration.
\newblock {\em Ann. Inst. H. Poincar\'{e} Anal. Non Lin\'{e}aire 37}, 5 (2020),
  1075--1108.

\bibitem{ModenaSzekelyhidi18}
{\sc Modena, S., and Sz\'{e}kelyhidi, Jr., L.}
\newblock Non-uniqueness for the transport equation with {S}obolev vector
  fields.
\newblock {\em Ann. PDE 4}, 2 (2018), Paper No. 18, 38.

\bibitem{MR3884855}
{\sc Modena, S., and Sz\'{e}kelyhidi, Jr., L.}
\newblock Non-uniqueness for the transport equation with {S}obolev vector
  fields.
\newblock {\em Ann. PDE 4}, 2 (2018), Paper No. 18, 38.

\bibitem{ModenaSzekelyhidi19}
{\sc Modena, S., and Sz\'{e}kelyhidi, Jr., L.}
\newblock Non-renormalized solutions to the continuity equation.
\newblock {\em Calc. Var. Partial Differential Equations 58}, 6 (2019), Paper
  No. 208, 30.

\bibitem{MR4029736}
{\sc Modena, S., and Sz\'{e}kelyhidi, Jr., L.}
\newblock Non-renormalized solutions to the continuity equation.
\newblock {\em Calc. Var. Partial Differential Equations 58}, 6 (2019), Paper
  No. 208, 30.

\bibitem{muller_sverak03}
{\sc M\"{u}ller, S., and {\v{S}}ver\'{a}k, V.}
\newblock Convex integration for {L}ipschitz mappings and counterexamples to
  regularity.
\newblock {\em Ann. of Math. (2) 157}, 3 (2003), 715--742.

\bibitem{nash54}
{\sc Nash, J.}
\newblock {$C^1$} isometric imbeddings.
\newblock {\em Ann. of Math. (2) 60\/} (1954), 383--396.

\bibitem{PitchoSorella21}
{\sc Pitcho, J., and Sorella, M.}
\newblock Almost everywhere non-uniqueness of integral curves for
  divergence-free sobolev vector fields, 2021.

\end{thebibliography}
\end{document}